\documentclass[a4paper,11pt,reqno]{amsart}

\usepackage[T1]{fontenc}
\usepackage{lmodern,mathrsfs}
\usepackage{amsmath,amssymb,amsthm,mathtools,geometry,hyperref,microtype,xcolor}
\allowdisplaybreaks[2]
\definecolor{DarkBlue}{rgb}{0.08,0.12,0.55}
\definecolor{DarkRed}{rgb}{0.55,0.08,0.08}
\hypersetup{colorlinks=true,linkcolor=DarkBlue,citecolor=DarkRed,urlcolor=DarkBlue,
 pdftitle={A Curvature Gap for Minimal Submanifolds in Spheres},
 pdfauthor={Fagui Li and Yuhang Zhao}}

\numberwithin{equation}{section}
\newtheorem{theorem}{Theorem}[section]
\newtheorem{proposition}[theorem]{Proposition}
\newtheorem{lemma}[theorem]{Lemma}

\newtheorem{remark}[theorem]{Remark}
\newtheorem*{acknow}{Acknowledgments}

\newcommand{\Sn}{\mathbb S}
\newcommand{\dd}{\,d\mu}
\newcommand{\ip}[2]{\langle #1,#2\rangle}

\title[A Curvature Gap for Minimal Submanifolds in Spheres]{A Curvature Gap for Minimal Submanifolds in Spheres}

\author[F. G. Li]{Fagui Li}
\address{Frontier Interdisciplinary Domain, Beijing Institute of Technology, Zhuhai, Guangdong 519088, P. R. China}
\email{lifagui@bitzh.edu.cn}

\author[Y. Zhao]{Yuhang Zhao${}^{*}$}
\address{School of Mathematics, Nanjing University, Nanjing 210093, P. R. China}
\email{yuhangzhao@smail.nju.edu.cn}

\subjclass[2020]{53C20, 53C24, 53C42}
\keywords{Minimal submanifold, sphere, higher codimension, Chern conjecture, second fundamental form, curvature gap}
\thanks{$^{*}$ Corresponding author.}
\thanks{F. G. Li is partially supported by NSFC (No. 12271040 and 12501061), the Guangdong Provincial Association for Science and Technology Youth Talent Support Program (No. SKXRC2026413) and the Research Start-up Funding of Beijing Institute of Technology (No. 5640011253301).}
\begin{document}

\begin{abstract}
Let $F:M^n\to\Sn^{n+q}(1)$ be a closed  minimal immersion in the
unit sphere, with $n\ge3$, $q\ge2$. Let $S$ be the squared length of its second fundamental
form.
We prove that if $M$ is not totally geodesic, then
$$
 \max_M S>
 \frac{n(288\sqrt6\,n-288)}{(288\sqrt6+192)n-137}
 >\frac{2n}{3}+\frac{31-9\sqrt6}{75}(n-2).
$$
\end{abstract}

\maketitle

\section{Introduction}
Let
$
 F:M^n\longrightarrow \Sn^{n+q}(1)
$
be a closed connected minimal immersion in the unit sphere, with second fundamental form $h$ and
$
 S=|h|^2.
$
The basic identity is Simons' formula
\begin{equation*}
 \frac12\Delta S
 =|\nabla h|^2+nS
 -\sum_{\alpha,\beta}\langle A_\alpha,A_\beta\rangle^2
 -\sum_{\alpha,\beta}|[A_\alpha,A_\beta]|^2.
\end{equation*}
For codimension $q\ge2$, the Li--Li matrix inequality~\cite{LiLi1992}, closely related to the DDVV inequality~\cite{GeTang2008,Lu2011}, gives
\begin{equation*}
 \sum_{\alpha,\beta}\langle A_\alpha,A_\beta\rangle^2
 +\sum_{\alpha,\beta}|[A_\alpha,A_\beta]|^2
 \le\frac32S^2,
\end{equation*}
and hence the universal first-pinching value $2n/3$.
The classical first-gap theorem of Simons, Chern--do Carmo--Kobayashi, Li--Li, and Chen--Xu says that $S\le2n/3$ forces the totally geodesic case, except for the two-dimensional Veronese equality model.  Thus, when $n\ge3$, the classical theorem gives the qualitative statement
\[
 S_{\max}:= \max_M S>\frac{2n}{3}.
\]
Our purpose is to make this strict inequality uniform and explicit, with a remainder that grows linearly in $n-2$ and is independent of the higher codimension $q\ge2$.

The pinching theory for minimal submanifolds in spheres originates in Simons~\cite{Simons1968}.  The equality cases of the first gap were analyzed by Lawson~\cite{Lawson1969} and Chern, do Carmo, and Kobayashi~\cite{ChernDoCarmoKobayashi1970}, while Yau~\cite{Yau1975} sharpened the codimension-two picture.  Li and Li~\cite{LiLi1992} and Chen and Xu~\cite{ChenXu1993} established the universal higher-codimensional threshold $2n/3$.  Lu~\cite{Lu2011} later introduced the object we shall call Lu's fundamental matrix, used its second eigenvalue to refine the higher-codimensional rigidity problem, and developed the optimal matrix inequalities underlying the Chern--Lu program.  The DDVV equality structure is closely related; see Ge and Tang~\cite{GeTang2008} and Lu~\cite{Lu2011}.  

To place the present result in the classical Chern program, recall that the Gauss equation gives
\begin{equation*}
 \operatorname{Scal}_M=n(n-1)-S
\end{equation*}
for a minimal immersion into the unit sphere.  Thus constant scalar curvature is equivalent to constant $S$.  In the hypersurface case, Chern's conjecture is commonly formulated as the discreteness of the possible constant scalar curvatures; a widely used refined form predicts that every closed minimal hypersurface with constant scalar curvature is isoparametric, as discussed for example by Ge and Tang~\cite{GeTangChern2010}, Xu and Xu~\cite{XuXu2017,XuXu2024}.  The first nonzero pinching value is $S=n$, realized by the Clifford hypersurfaces, and the classical second-gap problem asks in particular whether a constant value $S>n$ must satisfy $S\ge2n$.  Peng and Terng~\cite{PengTerng1983} initiated the second-gap analysis.  For the stronger pointwise pinching problem without assuming $S$ constant, Ding and Xin~\cite{DingXin2011} proved the gap $0\le S-n\le n/23$ in dimensions $n\ge6$; Xu and Xu~\cite{XuXu2017} improved this to $n/22$ in all dimensions; and Lei, Xu, and Xu~\cite{LeiXuXu2017} further improved the constant to $n/18$.

There has also been substantial recent activity around the Chern conjecture.  
The three-dimensional result of de Almeida and Brito~\cite{A-B} concerns
closed hypersurfaces in the unit sphere with constant mean curvature
and constant nonnegative scalar curvature. In dimensions $n\ge4$,
Tang, Wei, and Yan~\cite{TangWeiYan2020} proved isoparametricity under
these hypotheses, together with constancy of
$\operatorname{tr}(A^k)$ for $k=3,\ldots,n-1$ and the assumption that
the principal curvatures are everywhere distinct. Tang and
Yan~\cite{TangYan2023} removed the distinctness assumption while
retaining the other hypotheses. 
Cheng, Wei, and Yamashiro~\cite{ChengWeiYamashiro2021} obtained a strong rigidity range under the additional assumption that the third power sum of the principal curvatures is constant.  In dimension four, Li~\cite{LiChern2022} obtained a related isoparametricity result.  Deng, Gu, and Wei~\cite{DengGuWei2017} proved that closed Willmore minimal hypersurfaces in $\mathbb S^5(1)$ with constant scalar curvature are isoparametric.  For closed minimal Willmore hypersurfaces with constant scalar curvature in dimensions $n\ge5$ and with $S>n$, Ge, Tan, Yan, and Zhang~\cite{GeTanYanZhang2025} obtained a second-gap lower bound of the form $2n-3+o(1)$ as $n\to\infty$, and proved the full $2n$ conclusion under an additional inequality involving the fourth power sum $\operatorname{tr}(A^4)$.  More recently, He, Xu, and Zhao~\cite{HeXuZhao2026} proved that a closed minimal hypersurface $M^4\to\Sn^5$ with constant scalar curvature and constant third mean curvature is isoparametric; Ge, Liu, Luo, and Yan~\cite{GeLiuLuoYan2026} and Deng and Kou~\cite{DengKou2026} obtained corresponding rigidity results under constant Gauss--Kronecker curvature.  Tan, Tang, Xie, and Yan~\cite{TanTangXieYan2026} proved that, among closed embedded minimal hypersurfaces with constant $S$ and constant $f_3=\operatorname{tr}(A^3)$, the set of possible values of $S$ is locally finite, without fixing either the topology or the value of $f_3$.  These results reinforce the expected link between constant scalar curvature and isoparametricity.

Parallel to the Chern--Lu program, Simon's conjecture gives a distinguished sequence of two-dimensional curvature-gap problems between consecutive Calabi values.  Ding, Ge, and Li~\cite{DingGeLi2025} established pinching rigidity results for minimal surfaces in spheres.  Ding, Ge, and Li~\cite{DingGeLiThird2026} subsequently developed the third-gap problem and obtained positive pinching results throughout the third Simon interval $5/3\le S\le9/5$, including the endpoint regimes.  More recently, Ding, Ge, and Li~\cite{DingGeLiSimon2026} gave a proof of all gaps in the Simon conjecture in a preprint, together with further curvature-rigidity consequences.  These surface results concern the quantized sequence of curvature intervals determined by the Calabi models.  The present theorem addresses a different but related gap phenomenon: in every dimension $n\ge3$, it gives a uniform codimension-free separation immediately above the universal Li--Li~\cite{LiLi1992} and Chen--Xu~\cite{ChenXu1993} threshold $2n/3$, without requiring $S$ to be constant.

The higher-codimensional situation has a different algebraic first threshold.  Under a flat normal bundle the shape operators commute and the hypersurface scale $n$ reappears; Ge, Li, and Zhang~\cite{GeLiZhangFlat2026} recently proved an explicit constant-$S$ second gap of at least $n/87$ for closed submanifolds of dimension $n\ge3$ in this setting.  Without flatness, Lu~\cite{Lu2011} refined the problem using the second eigenvalue of the fundamental matrix, leading to another Chern-type gap problem.  Recent works by Ding, Ge, Li, and Yang~\cite{DingGeLiYang2026} and by Ge, Li, and Zhang~\cite{GeLiZhangLu2026} prove Lu's conjecture in several surface settings and settle the codimension-two constant case, while Li and Zhao~\cite{LiZhao2026} constructed flat-torus examples showing that the corresponding discreteness statement for Lu's refined quantity fails in codimension at least three.  This contrast illustrates why a gap depending only on $S$ and valid for arbitrary normal bundle is structurally different from both the flat-normal Chern problem and Lu's refined second-gap problem.

Very recently, Firester and Tsiamis~\cite{FiresterTsiamis2026} disproved the
higher-codimensional Chern--do Carmo--Kobayashi discreteness conjecture:
they establish counterexamples for all $n\ge3$ in codimension $q\ge4$, and
also for even $n\ge4$ in codimension $q\ge3$. Their result concerns the
higher-codimensional formulation of the conjecture and does not settle the
classical hypersurface case. This further emphasizes that higher codimension
exhibits phenomena absent from the hypersurface setting.
The theorem below is therefore complementary to, rather than a direct solution of, the classical Chern conjecture.  It does not address the hypersurface second gap from $n$ to $2n$.  Instead, for arbitrary higher codimension $q\ge2$ and without assuming $S$ constant or imposing a pointwise lower pinching hypothesis, it produces a uniform explicit separation immediately above the higher-codimensional Li--Li and Chen--Xu threshold $2n/3$.

The proof pairs the Ricci-commutation tensor with its normal-curvature
part and integrates by parts. The resulting identity involves only
first derivatives and combines with the $S$-weighted Simons identity.
The algebraic estimates are controlled by the Li--Li deficit.
Retaining the trace-free condition on the shape operators improves the
normal-curvature contraction estimate. Two elementary completions of
the square then give one parameter choice valid in every dimension
$n\ge3$.

The first version of this paper
established the explicit codimension-independent estimate
$
 S_{\max}\ge\frac{2n}{3}+2^{-2n^2-10n-30}n^{-2n-2}.
$
The second version
improved the remainder to $(n-2)/[6300(39n+8)]$.
In related recent work, Lei~\cite{Lei2026HigherCodim} and
Xu and Zhao~\cite{XuZhao2026Pinching} proved, respectively,
\[
 S_{\max}>\frac{2n}{3}+\frac{n}{500}-\frac1{180},
 \qquad
 S_{\max}>\frac{2n}{3}+\frac{n-2}{182}.
\]
The estimates for arbitrary symmetric matrix tuples in Proposition~\ref{prop:normal-tangents} already yield a uniform curvature gap. By further exploiting the trace-free property of the shape operators, the refinement below gives the sharper dimension-dependent estimate stated in Theorem~\ref{thm:main}.

\begin{theorem}\label{thm:main}
Let $n\ge3$ and $q\ge2$, and let
$F:M^n\to\Sn^{n+q}(1)$ be a closed  minimal immersion which is
not totally geodesic. Then
\begin{equation}\label{eq:main-gap}
 S_{\max}>
 \frac{n(288\sqrt6\,n-288)}{(288\sqrt6+192)n-137}
 >\frac{2n}{3}+\frac{31-9\sqrt6}{75}(n-2).
\end{equation}
\end{theorem}

\begin{remark}
The restriction $n\ge3$ is sharp for a statement that does not exclude an
equality model: when $n=2$, the Veronese surface has $S\equiv4/3$.
Moreover, Wang and Wang~\cite[Section~5.3.3]{WangWang2023} constructed a
one-parameter family of immersed minimal two-spheres in $\mathbb S^4$
deforming the double-covered Veronese surface, with non-Veronese members
satisfying $\max_{\mathbb S^2}|h|^2>4/3$ and
$\max_{\mathbb S^2}|h|^2\to4/3$ along the family. More explicitly,
their curvature formula gives, in the parameter $t$ of that family,
\[
 K_t(0)=1-\frac23e^{-4t},\qquad
 K_t(\infty)=1-\frac23e^{4t}.
\]
The Gauss equation $|h_t|^2=2(1-K_t)$ yields
$\max_{\mathbb S^2}|h_t|^2\ge(4/3)e^{4|t|}>4/3$ for $t\ne0$.
Smooth convergence on the compact domain as $t\to0$ gives uniform
convergence of $|h_t|^2$ to $4/3$. Thus no positive absolute gap
remains in dimension two even after the Veronese equality model is
excluded.
\end{remark}

\begin{remark}
Xu--Zhao~\cite{XuZhao2026Pinching} conjectured that every closed non-totally geodesic minimal submanifold
$
 M^n\to\mathbb S^{n+q}(1)$ with $n\ge3,\  q\ge2,
$
satisfies the sharp estimate
\[
 S_{\max}\ge\frac{2n}{3}+\frac{n}{3}=n.
\]
The value $n$ is sharp: a Clifford minimal hypersurface in
$\mathbb S^{n+1}(1)$, regarded as a minimal submanifold of
$\mathbb S^{n+q}(1)$ through a totally geodesic inclusion, satisfies
$S\equiv n$. Theorem~\ref{thm:main} gives a uniform quantitative step
toward this conjectural endpoint.
\end{remark}

Write
\[
 Q=\sum_{\alpha,\beta}\langle A_\alpha,A_\beta\rangle^2
    +\sum_{\alpha,\beta}|[A_\alpha,A_\beta]|^2,
 \qquad D=3S^2-2Q\ge0.
\]
We decompose the Ricci-commutation tensor as
$\Phi=\mathcal L+\mathcal C$, where $\mathcal L$ is linear and
$\mathcal C$ is cubic in the second fundamental form.
Section~\ref{sec:global} introduces the normal part $\mathcal N$ of
$\mathcal C$ and states the algebraic estimates for
$\langle\mathcal C,\mathcal N\rangle$, including the trace-free
refinement used in the main theorem.
Section~\ref{sec:main-proof} combines the corresponding
integration-by-parts identity with the $S$-weighted Simons identity
to prove both inequalities in \eqref{eq:main-gap}, with strict
inequalities in every dimension. Appendix~\ref{sec:normal-certificates}
contains all the algebraic proofs. The finite-support moment argument
and the dimension-indexed scalar estimates are retained for
Proposition~\ref{prop:normal-tangents}; the proof of the new trace-free
estimate uses only $v\ge u^2$ and two completions of the square.

\section{Geometric and analytic preliminaries}

\subsection{Notation and conventions}
Let
$
 F:M^n\longrightarrow \Sn^{n+q}(1)\subset\mathbb R^{n+q+1}
$
be a smooth immersion and let $g$ be the induced metric.  Around a fixed point of
$M$, choose a local orthonormal tangent frame $\{e_1,\ldots,e_n\}$ with dual
coframe $\{\omega_1,\ldots,\omega_n\}$ and a local orthonormal normal frame
$\{\nu_1,\ldots,\nu_q\}$.  Tangent indices $i,j,k,l,r,s$ range from $1$ to $n$,
and normal indices $\alpha,\beta,\gamma$ range from $1$ to $q$.

We write $\bar\nabla$, $\nabla$ and $\nabla^\perp$ for the Levi--Civita
connection of the unit sphere, the induced Levi--Civita connection of $M$, and
the normal connection, respectively.  Our sign convention for the second
fundamental form and the shape operators is
\begin{align*}
 \bar\nabla_XY&=\nabla_XY+h(X,Y),\\
 \bar\nabla_X\nu_\alpha&=-A_\alpha X+\nabla_X^\perp\nu_\alpha,
 \qquad
 \langle A_\alpha X,Y\rangle=\langle h(X,Y),\nu_\alpha\rangle.
\end{align*}
Thus
\begin{equation*}
 h(e_i,e_j)=\sum_{\alpha=1}^q h_{ij}^\alpha\nu_\alpha,
 \qquad
 (A_\alpha)_{ij}=h_{ij}^\alpha,
\end{equation*}
and every $A_\alpha$ is a real symmetric matrix.  Throughout, ``closed'' means compact without boundary.  We use the non-averaged
mean-curvature vector
\[
 H:=\sum_{i=1}^n h(e_i,e_i).
\]
Hence the immersion is minimal exactly when
\begin{equation*}
 H=0
 \quad\Longleftrightarrow\quad
 \sum_{i=1}^n h_{ii}^\alpha=\operatorname{tr}A_\alpha=0
 \quad(1\le\alpha\le q).
\end{equation*}
In particular, for a minimal immersion all shape operators lie in
$\operatorname{Sym}_0^2(\mathbb R^n)$, the space of trace-free real symmetric
$n\times n$ matrices.

All unadorned pointwise tensor norms are Hilbert--Schmidt norms induced by $g$
and the normal metric. Every displayed index sum runs over its full range;
in particular, $(k,l)$ is an ordered tangent pair.  For matrices we use
\[
 \langle A,B\rangle_F:=\operatorname{tr}(A^{\mathsf T}B)=\sum_{i,j}a_{ij}b_{ij},
 \qquad |A|^2=\langle A,A\rangle_F,
 \qquad [A,B]:=AB-BA.
\]
For matrix contractions we suppress the subscript $F$ when no ambiguity can arise.
Following Lu~\cite{Lu2011}, define Lu's fundamental matrix by
\[
 \mathcal A:=\bigl(\langle A_\alpha,A_\beta\rangle_F\bigr)_{\alpha,\beta=1}^q.
\]
We shall refer to $\mathcal A$ simply as the fundamental matrix below. It is the Gram
matrix of the shape operators with respect to the Frobenius inner product.
The squared length of the second fundamental form is therefore
\begin{equation*}
 S:=|h|^2
 =\sum_{\alpha=1}^q|A_\alpha|^2
 =\sum_{\alpha=1}^q\sum_{i,j=1}^n(h_{ij}^\alpha)^2.
\end{equation*}
Contracting the Gauss equation and using minimality gives
\begin{equation*}
 \operatorname{Scal}_M=n(n-1)-S.
\end{equation*}
In particular, on a minimal submanifold of the unit sphere, $S$ is constant if and only if the scalar curvature is constant.
Since $M$ is closed in the main theorem, we write
\[
 S_{\max}:=\max_M S.
\]

Covariant derivatives of $h$ are denoted by
\[
 h_{ij;k}^\alpha:=(\nabla h)_{ij;k}^\alpha,
 \qquad
 h_{ij;kl}^\alpha:=(\nabla^2h)_{ij;kl}^\alpha,
\]
and we also write $h_{ijk}^\alpha=h_{ij;k}^\alpha$. Thus
\[
 |\nabla h|^2=\sum_{\alpha,i,j,k}(h_{ij;k}^\alpha)^2.
\]
Our order convention is that, at a point where the tangent and normal frames
are synchronous,
\[
 h_{ij;k}^\alpha=(\nabla_{e_k}h)_{ij}^\alpha,
 \qquad
 h_{ij;kl}^\alpha=(\nabla_{e_l}\nabla_{e_k}h)_{ij}^\alpha;
\]
thus the rightmost derivative index is applied last.  This convention is used
in the Ricci commutation formula below.
At a point where both tangent and normal frames are synchronous,
$(\nabla_{e_k}A_\alpha)_{ij}=h_{ij;k}^\alpha$. Because the ambient sphere is a
space form, Codazzi gives
\begin{equation*}
 h_{ij;k}^\alpha=h_{ik;j}^\alpha;
\end{equation*}
combined with $h_{ij}^\alpha=h_{ji}^\alpha$, this implies that
$h_{ijk}^\alpha$ is symmetric in the three tangent indices.

We use the curvature-operator convention
\[
 R(X,Y)Z=\nabla_X\nabla_YZ-\nabla_Y\nabla_XZ
          -\nabla_{[X,Y]}Z,
\]
and the same convention for the normal connection. Our curvature components are
\[
 R_{ijkl}=\langle R(e_k,e_l)e_j,e_i\rangle,
 \qquad
 R^\perp_{\alpha\beta kl}
 =\langle R^\perp(e_k,e_l)\nu_\beta,\nu_\alpha\rangle.
\]
With these definitions, our curvature-sign convention is fixed by the
following Gauss and Ricci equations:
\begin{align}
 R_{ijkl}
 &=\delta_{ik}\delta_{jl}-\delta_{il}\delta_{jk}
 +\sum_{\gamma=1}^q
 \left(h_{ik}^\gamma h_{jl}^\gamma-h_{il}^\gamma h_{jk}^\gamma\right),
 \label{eq:Gauss-prelim}\\
 R^\perp_{\alpha\beta kl}
 &=\sum_{r=1}^n
 \left(h_{kr}^\alpha h_{lr}^\beta-h_{kr}^\beta h_{lr}^\alpha\right).
 \label{eq:Ricci-prelim}
\end{align}
With this convention, the Ricci commutation formula for the second fundamental
form is
\begin{equation}\label{eq:ricci-comm-prelim}
 h_{ij;kl}^\alpha-h_{ij;lk}^\alpha
 =\sum_{r=1}^n h_{rj}^\alpha R_{rikl}
 +\sum_{r=1}^n h_{ir}^\alpha R_{rjkl}
 +\sum_{\beta=1}^q h_{ij}^\beta R^\perp_{\beta\alpha kl}.
\end{equation}
These formulas fix all curvature signs used later.  The Laplace--Beltrami
operator is
\[
 \Delta f:=\sum_{i=1}^n f_{;ii},
\]
and the same trace convention is used for the rough Laplacian of tensors.  We write $d\mu$ for the volume measure of $(M,g)$.

\subsection{Simons' identity and the Li--Li deficit}
We use the nonnegative Li--Li deficit
\begin{equation}\label{eq:QD}
 \begin{split}
 Q&:=\sum_{\alpha,\beta}\langle A_\alpha,A_\beta\rangle^2
       +\sum_{\alpha,\beta}|[A_\alpha,A_\beta]|^2,\\
 D&:=3S^2-2Q\ge0.
 \end{split}
\end{equation}
These definitions also apply to arbitrary symmetric matrix tuples; the
nonnegativity is the Li--Li inequality~\cite{LiLi1992} and is also
proved directly in Appendix~\ref{sec:normal-certificates}.

\begin{lemma}[Simons' identity~\cite{Simons1968,Lu2011}]\label{lem:simons-lili}
For a minimal immersion into the unit sphere,
\begin{equation}\label{eq:simons}
 \frac12\Delta S
 =|\nabla h|^2+nS-Q
 =|\nabla h|^2+nS-\frac32S^2+\frac12D.
\end{equation}
\end{lemma}
\begin{proof}
We include the contraction calculation to check the signs against our
curvature conventions. Codazzi and differentiated minimality give
\[
 \sum_k h^\alpha_{ik;k}=\sum_k h^\alpha_{kk;i}=0.
\]
Commuting the last two derivatives in
$\sum_k h^\alpha_{ij;kk}=\sum_k h^\alpha_{ik;jk}$ therefore gives
\[
 \Delta h^\alpha_{ij}
 =\sum_{k,r}h^\alpha_{rk}R_{rijk}
  +\sum_{k,r}h^\alpha_{ir}R_{rkjk}
  +\sum_{k,\beta}h^\beta_{ik}R^\perp_{\beta\alpha jk}.
\]
The omitted term is $\sum_kh^\alpha_{ik;kj}=0$, the covariant
derivative of the preceding contracted identity. The constant-curvature
parts of the first two sums are $A_\alpha$ and $(n-1)A_\alpha$,
respectively. For each $\beta$, their cubic parts and the normal-curvature
part are, in the same order,
\[
 \begin{gathered}
 A_\beta A_\alpha A_\beta-\langle A_\alpha,A_\beta\rangle A_\beta,
 \qquad -A_\alpha A_\beta^2,\\
 A_\beta A_\alpha A_\beta-A_\beta^2A_\alpha.
 \end{gathered}
\]
Here $\operatorname{tr}A_\beta=0$ removes the trace term in the
second sum. Consequently,
\[
 \Delta A_\alpha
 =nA_\alpha-\sum_\beta\langle A_\alpha,A_\beta\rangle A_\beta
  -\sum_\beta[A_\beta,[A_\beta,A_\alpha]],
\]
where $\Delta A_\alpha$ denotes the $\alpha$-component of the rough
Laplacian of $h$, using both the tangent and normal connections.
For symmetric matrices, cyclicity of the trace gives
\[
 \langle A_\alpha,[A_\beta,[A_\beta,A_\alpha]]\rangle
 =|[A_\beta,A_\alpha]|^2.
\]
Taking the Frobenius product with $A_\alpha$ and summing over $\alpha$
yields $\langle h,\Delta h\rangle=nS-Q$. Now use
$\frac12\Delta|h|^2=|\nabla h|^2+\langle h,\Delta h\rangle$ and
\eqref{eq:QD}. This agrees with \cite[equation~(31)]{Lu2011}.
\end{proof}

For symmetric matrices $P,B$, we shall also use
\begin{equation}\label{eq:BW}
 |[P,B]|^2\le2|P|^2|B|^2.
\end{equation}
Indeed, diagonalizing $P$ gives
$[P,B]_{ij}=(\lambda_i-\lambda_j)b_{ij}$, and
$(\lambda_i-\lambda_j)^2\le2|P|^2$.

\section{A global Li--Li deficit estimate}\label{sec:global}

\subsection{The Ricci-commutation tensor}
For each ordered tangent pair $(k,l)$ define skew-symmetric matrices
\[
 (R_0^{kl})_{ij}=\delta_{ik}\delta_{jl}-\delta_{il}\delta_{jk},
 \qquad
 (G^{kl})_{ij}=\sum_\gamma
   \bigl((A_\gamma)_{ik}(A_\gamma)_{jl}
          -(A_\gamma)_{il}(A_\gamma)_{jk}\bigr).
\]
Set
\begin{equation}\label{eq:LC}
 \begin{aligned}
 \mathcal L_\alpha^{kl}&=[A_\alpha,R_0^{kl}],\\
 \mathcal C_\alpha^{kl}&=[A_\alpha,G^{kl}]
       +\sum_\beta A_\beta[A_\beta,A_\alpha]_{kl},\qquad
 \Phi=\mathcal L+\mathcal C.
 \end{aligned}
\end{equation}
Thus $\mathcal L$ is linear and $\mathcal C$ is cubic in the tuple.
The linear part has the explicit components
\[
 \mathcal L^\alpha_{ij kl}
 =h^\alpha_{kj}\delta_{il}-h^\alpha_{lj}\delta_{ik}
   +h^\alpha_{ik}\delta_{jl}-h^\alpha_{il}\delta_{jk}.
\]
For an actual immersion, \eqref{eq:Gauss-prelim}, \eqref{eq:Ricci-prelim},
and \eqref{eq:ricci-comm-prelim} give
\begin{equation}\label{eq:commutator}
 \Phi^\alpha_{ij kl}=h^\alpha_{ij;kl}-h^\alpha_{ij;lk}.
\end{equation}

\subsection{The normal-curvature contraction}\label{sec:normal-projection}
To estimate the Li--Li deficit globally, we test the cubic part in
\eqref{eq:LC} against its normal-curvature part
\[
 \mathcal N_\alpha^{kl}:=\sum_\beta A_\beta[A_\beta,A_\alpha]_{kl},
 \qquad J:=\sum_{\alpha,\beta}|[A_\alpha,A_\beta]|^2,
 \qquad \mathscr P:=\ip{\mathcal C}{\mathcal N}.
\]
These are invariant tensor contractions. More precisely, if
$\widehat A_\alpha=\sum_\gamma Q_{\alpha\gamma}P A_\gamma P^{\mathsf T}$
for $(P,Q)\in O(n)\times O(q)$, then $\mathcal C$ and $\mathcal N$
transform by the induced orthogonal action on their normal index and
four tangent indices. Their pairing $\mathscr P$ is thus frame independent.
The following pointwise algebraic estimates will be proved in
Appendix~\ref{sec:normal-certificates}.

\begin{proposition}\label{prop:normal-tangents}
For each of the following four parameter choices, every symmetric
matrix tuple satisfies
\begin{equation}\label{eq:normal-tangent}
 \mathscr P+\eta S(S^2-J)\ge\alpha S^3-\beta SD.
\end{equation}
The algebraic inequalities themselves impose no
restriction on the matrix size or the number of matrices, and require
no trace condition.
\[
\begin{array}{c|ccc}
\text{dimension used}&\alpha&\beta&\eta\\ \hline
n=3&3/2&529/150&71/75\\
n=4&299/200&7/2&3/5\\
n=5&37/25&69/20&1/2\\
n\ge6&13/9&10/3&0
\end{array}
\]
For the parameter choice corresponding to $n=3$, the residual
\[
 \mathscr E:=\mathscr P+\eta S(S^2-J)-\alpha S^3+\beta SD
\]
is strictly positive whenever $S>0$ and $D>0$. For $S=0$ it vanishes.
\end{proposition}
\begin{proof}
See Lemmas~\ref{lem:normalized-scalar-reduction} and
\ref{lem:moment-bounds}, together with
Subsections~\ref{subsec:small-dimensions} and~\ref{subsec:tail} of
Appendix~\ref{sec:normal-certificates}.
\end{proof}

For a minimal immersion, the shape operators are trace-free. This
additional information gives the estimate used below.

\begin{proposition}\label{prop:tracefree-tangent}
Let $n\ge3$ and let $(A_\alpha)$ be a finite tuple of real symmetric
trace-free $n\times n$ matrices. Define
\[
 \alpha_n:=\sqrt6-\frac43+\frac{137}{144n}.
\]
Then
\begin{equation}\label{eq:tracefree-tangent}
 \mathscr P\ge\alpha_nS^3-\sqrt6\,SD.
\end{equation}
The inequality is strict whenever $S>0$, and both sides vanish when
$S=0$. The number of matrices is unrestricted.
\end{proposition}
\begin{proof}
See Subsection~\ref{subsec:tracefree}. 
\end{proof}

\section{Proof of the curvature gap}\label{sec:main-proof}

Retain the decomposition $\Phi=\mathcal L+\mathcal C$ and the normal
part $\mathcal N$ from Section~\ref{sec:normal-projection}. Pairing
Ricci commutation with $\mathcal N$ and integrating by parts gives
an identity involving only first derivatives.

\begin{lemma}\label{lem:normal-integration}
Define $(B_{\alpha k})_{ij}:=h^\alpha_{ij;k}$, using the full
covariant derivative of the normal-valued tensor $h$, and put
\[
 (M_k)_{\alpha\beta}=\ip{A_\alpha}{B_{\beta k}},\qquad
 W=\sum_{k,\alpha,\beta}
   \ip{[A_\alpha,A_\beta]}{[B_{\alpha k},B_{\beta k}]}.
\]
For every closed minimal immersion,
\begin{equation}\label{eq:normal-exact}
 \int_M(\mathscr P-2J)\dd
 =\int_M\left(W+\sum_k|M_k-M_k^{\mathsf T}|^2\right)\dd,
\end{equation}
where \(M_k^{\mathsf T}\) denotes the transpose of \(M_k\).
Moreover, the following pointwise estimate holds:
\begin{equation*}
 \mathscr R:=3S|\nabla h|^2-W-\sum_k|M_k-M_k^{\mathsf T}|^2\ge0.
\end{equation*}
\end{lemma}
\begin{proof}
All contractions are invariant under orthogonal tangent and normal
changes of frame. We compute at a fixed but arbitrary point \(p\in M\), using local orthonormal tangent and normal frames that are synchronous at \(p\), so that the coefficients of the tangent Levi-Civita connection and of the normal connection vanish at \(p\). No normal frame diagonalizing the fundamental matrix is differentiated.
In this notation
$(B_{\alpha k})_{ij}=h^\alpha_{ij;k}$, and each $B_{\alpha k}$ is
symmetric. Codazzi and differentiated minimality give
\[
 h^\alpha_{ijk}=h^\alpha_{jik}=h^\alpha_{ikj},\qquad
 \sum_i h^\alpha_{iik}=0.
\]
Write $C_{\alpha\beta}=[A_\alpha,A_\beta]$ in this proof.
The explicit formula for $\mathcal L$ gives
$\langle\mathcal L,\mathcal N\rangle=-2J$.
Indeed, using the symmetry \(h^\beta_{ij}=h^\beta_{ji}\), we have
$$
\begin{aligned}
\sum_{i,j}
\bigl(h^\alpha_{kj}\delta_{il}
      -h^\alpha_{lj}\delta_{ik}\bigr)h^\beta_{ij}
&=
\sum_j h^\alpha_{kj}h^\beta_{lj}
-\sum_j h^\alpha_{lj}h^\beta_{kj}  \\
&=(A_\alpha A_\beta-A_\beta A_\alpha)_{kl}
=(C_{\alpha\beta})_{kl}.
\end{aligned}
$$
Hence, after pairing with
\((C_{\beta\alpha})_{kl}=-(C_{\alpha\beta})_{kl}\)
and summing over \(\alpha,\beta,k,l\), these two terms contribute
$$
-\sum_{\alpha,\beta}|C_{\alpha\beta}|^2=-J.
$$
Similarly,
$
\sum_{i,j}
\bigl(h^\alpha_{ik}\delta_{jl}
      -h^\alpha_{il}\delta_{jk}\bigr)h^\beta_{ij}
=(C_{\alpha\beta})_{kl},
$
so the remaining two terms give another \(-J\). Therefore
$$
\langle\mathcal L,\mathcal N\rangle=-2J.
$$

The tensor $\mathcal N$ is skew in its last two indices. Thus
\eqref{eq:commutator} and integration by parts imply
\[
\begin{split}
 \int_M\langle\Phi,\mathcal N\rangle\dd
 &=2\int_M\sum_{\alpha,i,j,k,l}
          h^\alpha_{ij;kl}\mathcal N_{\alpha ijkl}\dd\\
 &=-2\int_M\sum_{\alpha,i,j,k,l}
          h^\alpha_{ijk}\nabla_l\mathcal N_{\alpha ijkl}\dd.
\end{split}
\]
There is no boundary term because $M$ is closed. The covariant product
rule applies to the full normal-valued tensor. The following identities are covariant and hence valid in arbitrary local orthonormal tangent and normal frames. The synchronous choice at \(p\) is used only to make the corresponding tangent and normal connection coefficients vanish at the point of calculation.
\[
 \nabla_l\mathcal N_{\alpha ijkl}
 =\sum_\beta h^\beta_{ijl}(C_{\beta\alpha})_{kl}
  +\sum_\beta h^\beta_{ij}\nabla_l(C_{\beta\alpha})_{kl}.
\]
The contraction from the first summand is
\[
\begin{split}
 -2\sum_{\alpha,\beta,i,j,k,l}
       h^\alpha_{ijk}h^\beta_{ijl}(C_{\beta\alpha})_{kl}
 &=2\sum_{\alpha,\beta,i}
           \ip{C_{\alpha\beta}}{B_{\alpha i}B_{\beta i}}\\
 &=\sum_{\alpha,\beta,i}
           \ip{C_{\alpha\beta}}{[B_{\alpha i},B_{\beta i}]}=W.
\end{split}
\]
The first equality uses full symmetry of $h^\alpha_{ijk}$.
The second uses $\langle C,UV\rangle=\frac12\langle C,[U,V]\rangle$
for skew $C$ and symmetric $U,V$.

For the second summand of the product rule, expand the divergence:
\[
\begin{split}
 \sum_l\nabla_l(C_{\beta\alpha})_{kl}
 =\sum_{r,l}\bigl(&h^\beta_{krl}h^\alpha_{rl}
            +h^\beta_{kr}h^\alpha_{rll}
            -h^\alpha_{krl}h^\beta_{rl}
            -h^\alpha_{kr}h^\beta_{rll}\bigr).
\end{split}
\]
The second and fourth sums vanish, since
$\sum_l h^\alpha_{rll}=\sum_l h^\alpha_{llr}=0$ and similarly for
$\beta$. The remaining terms are
\[
 \sum_l\nabla_l(C_{\beta\alpha})_{kl}
   =(M_k)_{\alpha\beta}-(M_k)_{\beta\alpha}.
\]
Their contraction is therefore
\[
 -2\sum_{\alpha,\beta,k}(M_k)_{\beta\alpha}
        \bigl((M_k)_{\alpha\beta}-(M_k)_{\beta\alpha}\bigr)
   =\sum_k|M_k-M_k^{\mathsf T}|^2.
\]
Here interchange of $\alpha,\beta$ identifies the sums of the squares
of the two transposed entries. Since
$\langle\Phi,\mathcal N\rangle=\mathscr P-2J$, this proves
\eqref{eq:normal-exact}.

We next bound its two first-derivative terms. The DDVV estimate
$J\le S^2$, proved in \eqref{eq:ddvv-new} (see also \cite{GeTang2008,Lu2011}), also applies to the
symmetric tuple $(B_{\alpha k})_\alpha$ for each fixed $k$.
Consequently,
\[
\begin{split}
 W&\le\sum_k\sqrt J
        \left(\sum_{\alpha,\beta}
                 |[B_{\alpha k},B_{\beta k}]|^2\right)^{1/2}\le S\sum_{k,\alpha}|B_{\alpha k}|^2=S|\nabla h|^2.
\end{split}
\]
For real rectangular matrices \(U\) and \(V\) of the same size, we use
$$
 |U^{\mathsf T}V-V^{\mathsf T}U|^2
 \le 2|U|^2|V|^2.
$$
Indeed,
if \(U\) and \(V\) have only one column, the left-hand side vanishes, so
the estimate is immediate. Otherwise, let \(K\) be a skew-symmetric
matrix with Frobenius norm \(|K|=1\), and denote by
$$
 |K|_{\mathrm{op}}:=\sup_{|x|=1}|Kx|
$$
its operator norm, equivalently its largest singular value. The nonzero
singular values of a real skew-symmetric matrix occur in equal pairs.
Hence the normalization \(|K|=1\) implies
$$
 |K|_{\mathrm{op}}\le\frac1{\sqrt2}.
$$
Using \(K^{\mathsf T}=-K\) and the Frobenius inner product, we obtain
$$
\begin{aligned}
\left|\left\langle
 U^{\mathsf T}V-V^{\mathsf T}U,K
\right\rangle\right|
&=2|\langle V,UK\rangle|\\
&\le 2|V|\,|UK|\\
&\le 2|U|\,|V|\,|K|_{\mathrm{op}}\\
&\le \sqrt2\,|U|\,|V|.
\end{aligned}
$$
Since \(U^{\mathsf T}V-V^{\mathsf T}U\) is skew-symmetric, its
Frobenius norm is the supremum of its Frobenius pairings with
skew-symmetric matrices \(K\) satisfying \(|K|=1\). 
For each fixed \(k\), apply this estimate to the matrices \(U\) and
\(V\) whose \(\alpha\)-th columns are the vectorizations of
\(A_\alpha\) and \(B_{\alpha k}\), respectively. Then
$$
 |U|^2=\sum_\alpha|A_\alpha|^2=S,
 \qquad
 |V|^2=\sum_\alpha|B_{\alpha k}|^2,
$$
while
$$
 (U^{\mathsf T}V)_{\alpha\beta}
 =\langle A_\alpha,B_{\beta k}\rangle
 =(M_k)_{\alpha\beta}.
$$
Consequently,
$$
 U^{\mathsf T}V-V^{\mathsf T}U
 =M_k-M_k^{\mathsf T},
$$
and therefore
$$
 |M_k-M_k^{\mathsf T}|^2
 \le
 2S\sum_\alpha|B_{\alpha k}|^2.
$$
Summing over \(k\) yields
$$
 \sum_k|M_k-M_k^{\mathsf T}|^2
 \le2S|\nabla h|^2.
$$
Together with \(W\le S|\nabla h|^2\), this gives
$$
 \mathscr R
 =3S|\nabla h|^2-W-\sum_k|M_k-M_k^{\mathsf T}|^2
 \ge0.
$$
All integrations above are with respect to the Riemannian volume
density, so no orientability assumption on \(M\) is required.
\end{proof}

\begin{proof}[\textbf{Proof of Theorem~\ref{thm:main}}]
By Lemma~\ref{lem:simons-lili}, multiplying \eqref{eq:simons} by
$2S$ and integrating by parts gives
\begin{equation}\label{eq:normal-weighted-simons}
 \int_M SD\dd
 =3\int_M S^3\dd-2n\int_M S^2\dd
  -2\int_M S|\nabla h|^2\dd-\int_M|\nabla S|^2\dd.
\end{equation}
Indeed, $S\Delta S=2S|\nabla h|^2+2nS^2-3S^3+SD$ and
$\int_M S\Delta S\dd=-\int_M|\nabla S|^2\dd$.
For $\alpha_n$ from Proposition~\ref{prop:tracefree-tangent}, put
\[
 \mathscr E_n:=\mathscr P-\alpha_nS^3+\sqrt6\,SD.
\]
By Proposition~\ref{prop:tracefree-tangent},
$\mathscr E_n\ge0$ everywhere and $\mathscr E_n>0$ wherever $S>0$.
Lemma~\ref{lem:normal-integration} gives
\[
 \int_M\mathscr P\dd
 =2\int_M J\dd+3\int_M S|\nabla h|^2\dd
  -\int_M\mathscr R\dd.
\]
Substitute this and \eqref{eq:normal-weighted-simons} into
$\int_M\mathscr E_n\dd$. Writing $2J=2S^2-2(S^2-J)$ and
rearranging yields
\begin{equation}\label{eq:normal-final-integral}
\begin{split}
 &\qquad\int_M S^2\bigl((3\sqrt6-\alpha_n)S-2n\sqrt6+2\bigr)\dd\\
 &\quad=(2\sqrt6-3)\int_M S|\nabla h|^2\dd
      +\sqrt6\int_M|\nabla S|^2\dd\\
 &\qquad\quad+2\int_M(S^2-J)\dd
      +\int_M\mathscr E_n\dd+\int_M\mathscr R\dd.
\end{split}
\end{equation}
Every term on the right is nonnegative, by $2\sqrt6-3>0$,
the DDVV inequality $J\le S^2$, and
Lemma~\ref{lem:normal-integration}. Moreover, $h\not\equiv0$
implies that $\{S>0\}$ is a nonempty open set, hence has positive
Riemannian measure. Therefore
\[
 \int_M\mathscr E_n\dd>0.
\]

On the other hand,
\[
 3\sqrt6-\alpha_n=2\sqrt6+\frac43-\frac{137}{144n}>0
 \qquad(n\ge3).
\]
If
\[
 S_{\max}\le
 \frac{2n\sqrt6-2}{3\sqrt6-\alpha_n}
 =\frac{n(288\sqrt6\,n-288)}{(288\sqrt6+192)n-137},
\]
the left-hand side of \eqref{eq:normal-final-integral} is nonpositive,
contradicting the strict positivity of its right-hand side. This proves the first inequality in \eqref{eq:main-gap}, including the endpoint. No division by $S$ is
used, so zeros of the second fundamental form cause no difficulty.

To obtain the uniform coefficient, direct subtraction gives
\[
\begin{split}
 &\frac{n(288\sqrt6\,n-288)}{(288\sqrt6+192)n-137}
   -\frac{2n}{3}-\frac{31-9\sqrt6}{75}(n-2)\\
 &\quad=
 \frac{(13167\sqrt6-29703)n+2466\sqrt6-8494}
      {75\bigl((288\sqrt6+192)n-137\bigr)}>0.
\end{split}
\]
The denominator is positive for $n\ge3$. Since
$\sqrt6>22/9$ (equivalently $486>484$), the numerator is strictly
greater than
\[
 2483n-2466\ge4983>0.
\]
Thus the second inequality in \eqref{eq:main-gap} follows. Both estimates hold without a
pointwise lower bound on $S$, constancy of $S$, or flatness of the
normal bundle.
\end{proof}

\appendix

\section{Proof of Propositions ~\ref{prop:normal-tangents} and ~\ref{prop:tracefree-tangent}}
\label{sec:normal-certificates}

Lemmas~\ref{lem:normalized-scalar-reduction} and~\ref{lem:moment-bounds}
provide the two algebraic ingredients used in the proof of
Proposition~\ref{prop:normal-tangents}: the scalar reduction and the
moment estimates for $u-v$. These are applied in
Subsections~\ref{subsec:small-dimensions} and~\ref{subsec:tail}
to verify the four parameter choices. No trace condition is required
in Subsections~\ref{subsec:normalized}--\ref{subsec:tail}.
Subsection~\ref{subsec:tracefree} proves
Proposition~\ref{prop:tracefree-tangent}, where trace-freeness is essential.
Throughout, the fundamental matrix is diagonalized only pointwise, and
the resulting normal frame is never differentiated.

\subsection{Normalized scalar reduction}
\label{subsec:normalized}
\begin{lemma}\label{lem:normalized-scalar-reduction}
Let $(A_\alpha)$ be a finite tuple of real symmetric matrices. If $S>0$,
normalize by replacing $A_\alpha$ with $A_\alpha/\sqrt S$, diagonalize the
fundamental matrix by an orthogonal normal change of frame, omit its zero
eigenvalues, and write
\[
 x_\alpha=|A_\alpha|^2>0,\qquad \sum_\alpha x_\alpha=1,\qquad
 u=\sum_\alpha x_\alpha^2,\qquad v=\sum_\alpha x_\alpha^3,
\]
\[
 U_{\alpha\beta}=|[A_\alpha,A_\beta]|^2,\qquad
 p=\sum_{\alpha<\beta}(x_\alpha+x_\beta)U_{\alpha\beta}.
\]
For the normalized tuple one has
\[
 J\le1,\qquad \mathscr P\ge-1,\qquad
 D=3-2u-2J\ge0.
\]
If $D<1$ and
\[
 a:=\frac{1-D}{2},
\]
then the fundamental matrix has at least two positive eigenvalues,
$u-v>0$, and
\begin{equation}\label{eq:scalar-reduction-new}
 \mathscr P\ge u-1+\frac{2a^2}{u-v},\qquad
 0<a\le\min\{u,1/2\},\qquad J=1+a-u.
\end{equation}
Moreover,
\begin{equation*}
 p\ge a,\qquad
 \sum_{\alpha<\beta}
 \frac{(x_\alpha+x_\beta)U_{\alpha\beta}^2}{x_\alpha x_\beta}
 \ge\frac{p^2}{u-v}.
\end{equation*}
In particular, $u<1$. All these assertions are pointwise and remain valid
when positive eigenvalues of the fundamental matrix are repeated.
\end{lemma}

\begin{proof}
If $S=0$, every matrix and every tensor under consideration vanishes.
Assume $S>0$ and replace $A_\alpha$ by $A_\alpha/\sqrt S$.
Throughout the scalar argument we keep the same notation for the
normalized tuple and its invariants, so $S=1$. In particular, the
normalized $J$, $D$, and $\mathscr P$ are the original quantities divided
by $S^2$, $S^2$, and $S^3$, respectively. Indeed, $J$ and $D$ have degree four in the matrices, whereas
$\mathscr P=\langle\mathcal C,\mathcal N\rangle$ and both terms
$S(S^2-J)$ and $SD$ have degree six. Multiplying the normalized
inequalities by the original $S^3$ will therefore recover the stated
homogeneous inequalities.

Diagonalize the fundamental matrix $\mathcal A$ by an orthogonal normal change of frame and omit
its zero eigenvalues, which correspond to zero matrices. Write
\[
 x_\alpha=|A_\alpha|^2>0,\qquad \sum_\alpha x_\alpha=1,\qquad
 u=\sum_\alpha x_\alpha^2,\qquad v=\sum_\alpha x_\alpha^3,
 \qquad U_{\alpha\beta}=|[A_\alpha,A_\beta]|^2.
\]
Thus $\langle A_\alpha,A_\beta\rangle=0$ for $\alpha\ne\beta$,
and $J=\sum_{\alpha,\beta}U_{\alpha\beta}$. For $\alpha\ne\beta$, set
\[
 e_{\alpha\beta}
 =\left(\frac{U_{\alpha\beta}}{x_\alpha x_\beta}-1\right)_+,
 \qquad e_{\alpha\alpha}=0,
\]
where $z_+=\max\{z,0\}$.
Fix $\alpha$ and let $\mathcal I_\alpha$ be the set of indices $\beta$
for which $e_{\alpha\beta}>0$. The matrices
$A_\beta/\sqrt{x_\beta}$, $\beta\in\mathcal I_\alpha$, form an
orthonormal family for the Frobenius inner product. If this family has
$m>0$ members, diagonalize $A_\alpha/\sqrt{x_\alpha}$ by a common
orthogonal tangent change. This preserves the Frobenius orthonormality
of the other matrices. Lu's weighted inequality
\cite[Lemma~2]{Lu2011} states that if $P$ is a diagonal symmetric
matrix with $|P|=1$ and $B_1,\ldots,B_m$ are symmetric matrices such
that $P,B_1,\ldots,B_m$ are pairwise Frobenius-orthogonal, then
\[
 \sum_{j=1}^m|[P,B_j]|^2
 \le\sum_{j=1}^m|B_j|^2+\max_{1\le j\le m}|B_j|^2.
\]
No trace-free assumption is required.  In our application,
$P=A_\alpha/\sqrt{x_\alpha}$ and the matrices
$A_\beta/\sqrt{x_\beta}$, $\beta\in\mathcal I_\alpha$, all
have unit Frobenius norm and are mutually orthogonal, including
orthogonality to $P$. Hence
\[
 \sum_{\beta\in\mathcal I_\alpha}
 \frac{U_{\alpha\beta}}{x_\alpha x_\beta}\le m+1.
\]
Subtracting $m$, and treating an empty family by the empty-sum convention,
proves
\[
 \sum_\beta e_{\alpha\beta}\le1.
\]
Since $e_{\alpha\beta}=e_{\beta\alpha}$, we obtain
\begin{equation}\label{eq:ddvv-new}
\begin{split}
 J&\le1-u+\sum_{\alpha\ne\beta}x_\alpha x_\beta e_{\alpha\beta}\\
  &\le1-u+\frac12\sum_{\alpha\ne\beta}
               (x_\alpha^2+x_\beta^2)e_{\alpha\beta}\\
  &=1-u+\sum_\alpha x_\alpha^2\sum_\beta e_{\alpha\beta}\le1.
\end{split}
\end{equation}
Here and below all sums over $\alpha\ne\beta$ count ordered pairs.
By homogeneity, \eqref{eq:ddvv-new} is the DDVV bound $J\le S^2$
for every symmetric tuple.

Put
\[
 p=\sum_{\alpha<\beta}(x_\alpha+x_\beta)U_{\alpha\beta}.
\]
We next prove $p\ge u+J-1$. For two distinct indices the exact identity
\[
\begin{split}
 &\frac{x_\alpha^2(1-x_\alpha)+x_\beta^2(1-x_\beta)}2
  -x_\alpha x_\beta\left(1-\frac{x_\alpha+x_\beta}{2}\right)\\
 &\qquad=\frac{(1-x_\alpha-x_\beta)(x_\alpha-x_\beta)^2}{2}\ge0
\end{split}
\]
uses $x_\alpha+x_\beta\le1$. Moreover,
\[
\begin{split}
 p-J+1-u
 &=\sum_{\alpha\ne\beta}
 \left[x_\alpha x_\beta-
       \left(1-\frac{x_\alpha+x_\beta}{2}\right)U_{\alpha\beta}\right]\\
 &\ge u-v-\sum_{\alpha\ne\beta}e_{\alpha\beta}x_\alpha x_\beta
                \left(1-\frac{x_\alpha+x_\beta}{2}\right)\\
 &\ge u-v-\sum_\alpha x_\alpha^2(1-x_\alpha)
                    \sum_\beta e_{\alpha\beta}\ge0.
\end{split}
\]
In the second line we used
$U_{\alpha\beta}\le x_\alpha x_\beta(1+e_{\alpha\beta})$ and
$\frac12\sum_{\alpha\ne\beta}x_\alpha x_\beta(x_\alpha+x_\beta)=u-v$.
The third line follows from the preceding scalar identity and the symmetry
of $e$. On the other hand, \eqref{eq:BW} implies
\[
 p\le2\sum_{\alpha<\beta}(x_\alpha+x_\beta)x_\alpha x_\beta
   =2\sum_\alpha x_\alpha^2(1-x_\alpha)
   \le\frac12\sum_\alpha x_\alpha=\frac12,
\]
since $x_\alpha(1-x_\alpha)\le1/4$. Combining the two bounds on $p$ gives
\[
 u+J\le p+1\le\frac32,\qquad D=3-2u-2J\ge0.
\]
By homogeneity this also proves the Li--Li deficit inequality \eqref{eq:QD} for the
original tuple.

For $C_{\alpha\beta}=[A_\alpha,A_\beta]$, the exact tensor contractions are
\begin{equation}\label{eq:normal-contraction-new}
 \begin{split}
 |\mathcal N|^2&=p,\\
 \mathscr P&=p-2\sum_{\alpha,\beta,\gamma}
       \ip{A_\gamma C_{\alpha\beta}A_\gamma}{C_{\alpha\beta}}.
 \end{split}
\end{equation}
Indeed, orthogonality in the resulting normal frame gives
\[
 |\mathcal N|^2
 =\sum_{\alpha,\beta}x_\beta U_{\alpha\beta}
 =\sum_{\alpha<\beta}(x_\alpha+x_\beta)U_{\alpha\beta}.
\]
For fixed \(\alpha,\beta,k,l\), substituting the definition of \(G\) gives
$$
\begin{aligned}
&\sum_{i,j,r}
(A_\beta)_{ij}(A_\alpha)_{rj}G_{rikl}\\
&\quad=
\sum_{\gamma,i,j,r}
(A_\beta)_{ij}(A_\alpha)_{rj}
\Bigl(
(A_\gamma)_{rk}(A_\gamma)_{il}
-(A_\gamma)_{rl}(A_\gamma)_{ik}
\Bigr)\\
&\quad=
\sum_\gamma
\Bigl[
(A_\gamma A_\alpha A_\beta A_\gamma)_{kl}
-
(A_\gamma A_\alpha A_\beta A_\gamma)_{lk}
\Bigr].
\end{aligned}
$$
Since \(A_\alpha,A_\beta,A_\gamma\) are symmetric,
$$
(A_\gamma A_\alpha A_\beta A_\gamma)_{lk}
=
(A_\gamma A_\beta A_\alpha A_\gamma)_{kl}.
$$
Therefore
$$
\begin{aligned}
\sum_{i,j,r}
(A_\beta)_{ij}(A_\alpha)_{rj}G_{rikl}
&=
\sum_\gamma
\bigl(
A_\gamma A_\alpha A_\beta A_\gamma
-
A_\gamma A_\beta A_\alpha A_\gamma
\bigr)_{kl}\\
&=
\sum_\gamma
(A_\gamma C_{\alpha\beta}A_\gamma)_{kl}.
\end{aligned}
$$
By the definition
\[
 \mathcal C_\alpha^{kl}=[A_\alpha,G^{kl}]+\mathcal N_\alpha^{kl},
\]
we have
\[
 \mathscr P=\ip{\mathcal C}{\mathcal N}
 =\sum_{\alpha,k,l}\ip{[A_\alpha,G^{kl}]}{\mathcal N_\alpha^{kl}}
  +|\mathcal N|^2.
\]
Thus it remains to compute the first term. Since
$$
[A_\alpha,G^{kl}]_{ij}
=
\sum_r
\bigl(
(A_\alpha)_{ir}G_{rjkl}
-
G_{irkl}(A_\alpha)_{rj}
\bigr),
$$
and \(G_{irkl}=-G_{rikl}\), the second term contributes
$$
\sum_{i,j,r}
(A_\beta)_{ij}(A_\alpha)_{rj}G_{rikl}.
$$
For the first term, interchanging \(i\) and \(j\), and using the symmetry
of \(A_\alpha\) and \(A_\beta\), gives the same expression:
$$
\begin{aligned}
\sum_{i,j,r}
(A_\beta)_{ij}(A_\alpha)_{ir}G_{rjkl}
&=
\sum_{i,j,r}
(A_\beta)_{ji}(A_\alpha)_{jr}G_{rikl}\\
&=
\sum_{i,j,r}
(A_\beta)_{ij}(A_\alpha)_{rj}G_{rikl}.
\end{aligned}
$$
Hence the two terms in the commutator \([A_\alpha,G^{kl}]\) give equal
contributions.
Since
$$
\mathcal N_{\alpha ijkl}
=
\sum_\beta
(A_\beta)_{ij}(C_{\beta\alpha})_{kl},
\qquad
C_{\beta\alpha}=-C_{\alpha\beta},
$$
we obtain
$$
\begin{aligned}
\sum_{\alpha,k,l}\ip{[A_\alpha,G^{kl}]}{\mathcal N_\alpha^{kl}}
&=
2\sum_{\alpha,\beta,k,l}
\left[
\sum_\gamma
(A_\gamma C_{\alpha\beta}A_\gamma)_{kl}
\right]
(C_{\beta\alpha})_{kl}\\
&=
-2\sum_{\alpha,\beta,\gamma}
\left\langle
A_\gamma C_{\alpha\beta}A_\gamma,
C_{\alpha\beta}
\right\rangle.
\end{aligned}
$$
Combining this with
$$
|\mathcal N|^2=p
$$
yields
$$
\mathscr P
=
p-
2\sum_{\alpha,\beta,\gamma}
\left\langle
A_\gamma C_{\alpha\beta}A_\gamma,
C_{\alpha\beta}
\right\rangle,
$$
which proves \eqref{eq:normal-contraction-new}.

Fix $\alpha<\beta$ and diagonalize $A_\alpha$ with eigenvalues
$\lambda_i$. Write $A_\beta=(b_{ij})$. For $i\ne j$,
\[
 (C_{\alpha\beta})_{ij}=(\lambda_i-\lambda_j)b_{ij},\qquad
 \lambda_i\lambda_j\le
       \frac{x_\alpha-(\lambda_i-\lambda_j)^2}{2}.
\]
The diagonal entries of $C_{\alpha\beta}$ are zero. Hence
\[
\begin{split}
 \ip{A_\alpha C_{\alpha\beta}A_\alpha}{C_{\alpha\beta}}
 &\le\frac{x_\alpha U_{\alpha\beta}}2
       -\frac12\sum_{i,j}(\lambda_i-\lambda_j)^4b_{ij}^2\\
 &\le\frac{x_\alpha U_{\alpha\beta}}2
       -\frac{U_{\alpha\beta}^2}{2x_\beta}.
\end{split}
\]
The last inequality follows from
\[
 \left(\sum_{i,j}(\lambda_i-\lambda_j)^2b_{ij}^2\right)^2
 \le \left(\sum_{i,j}b_{ij}^2\right)
       \left(\sum_{i,j}(\lambda_i-\lambda_j)^4b_{ij}^2\right).
\]
Interchanging $\alpha,\beta$ gives the analogous estimate. For any other
$\gamma$, diagonalizing $A_\gamma$ gives the upper bound
$x_\gamma U_{\alpha\beta}/2$, because an off-diagonal eigenvalue
product is at most $|A_\gamma|^2/2$.
Substituting these estimates into \eqref{eq:normal-contraction-new}, with
both orders of each normal pair retained, yields
\[
 \mathscr P\ge p-J+
 2\sum_{\alpha<\beta}
       \frac{(x_\alpha+x_\beta)U_{\alpha\beta}^2}{x_\alpha x_\beta}.
\]
If at least two eigenvalues of the fundamental matrix are positive, then
\[
 u-v=\sum_\alpha x_\alpha^2(1-x_\alpha)>0,
\]
and weighted Cauchy--Schwarz gives
\[
\begin{split}
 p^2
 &=\left(\sum_{\alpha<\beta}
 \sqrt{(x_\alpha+x_\beta)x_\alpha x_\beta}\,
 \sqrt{\frac{x_\alpha+x_\beta}{x_\alpha x_\beta}}
 U_{\alpha\beta}\right)^2\\
 &\le(u-v)\sum_{\alpha<\beta}
       \frac{(x_\alpha+x_\beta)U_{\alpha\beta}^2}{x_\alpha x_\beta}.
\end{split}
\]
Therefore
\[
 \mathscr P\ge p-J+\frac{2p^2}{u-v}\ge p-J\ge-1.
\]
If the fundamental matrix has only one positive eigenvalue, then $u=1$, $J=0$,
$\mathcal N=0$, and $D=1$, so $\mathscr P=0\ge-1$ directly.
This is the only case in which $u-v=0$: each retained weight is
positive and at most one, so
$\sum_\alpha x_\alpha^2(1-x_\alpha)=0$ forces every retained
$x_\alpha$ to equal one, and their sum is one.

When $D<1$, put $a=(1-D)/2$. Then $0<a\le1/2$, $p\ge a$,
$J=1+a-u$, and $u\ge a$ by \eqref{eq:ddvv-new}. For fixed $J,u,v$,
the expression $p-J+2p^2/(u-v)$ increases with $p\ge0$. Since $p\ge a$
and $J=1+a-u$, this proves \eqref{eq:scalar-reduction-new}. In particular
$u<1$, since $u=1$ would force the fundamental matrix to have rank one
and hence $D=1$, contrary to the present assumption $D<1$. Repeated positive
eigenvalues of the fundamental matrix cause no difficulty, since $\mathcal A$ is
diagonalized only pointwise and the resulting normal frame is never differentiated.

\end{proof}

\subsection{Moment estimates}

\begin{lemma}\label{lem:moment-bounds}
Let $x_\alpha>0$ satisfy $\sum_\alpha x_\alpha=1$, and set
\[
 u=\sum_\alpha x_\alpha^2,\qquad v=\sum_\alpha x_\alpha^3.
\]
Then
\begin{equation}\label{eq:coarse-moment-new}
 u-v\le u(1-u).
\end{equation}
Moreover, if $1/3\le u\le1/2$, then
\begin{equation}\label{eq:sharp-moment-new}
 u-v\le R(u):=\frac{8+(6u-2)^{3/2}}{36},
\end{equation}
and the function $1/R$ is concave on $[1/3,1/2]$.
\end{lemma}

\begin{proof}
Cauchy--Schwarz gives
\[
 u^2=\left(\sum_\alpha x_\alpha^{1/2}x_\alpha^{3/2}\right)^2
 \le\left(\sum_\alpha x_\alpha\right)
      \left(\sum_\alpha x_\alpha^3\right)=v.
\]
Hence $u-v\le u-u^2=u(1-u)$, proving
\eqref{eq:coarse-moment-new}.

We next prove \eqref{eq:sharp-moment-new}. Fix a finite number $m$ of
entries and a feasible value $u\in(1/3,1/2)$. The set
\[
 \left\{(x_1,\ldots,x_m)\in[0,1]^m:
       \sum_{\alpha=1}^m x_\alpha=1,\quad
       \sum_{\alpha=1}^m x_\alpha^2=u\right\}
\]
is compact, so $\sum_{\alpha=1}^m x_\alpha^3$ attains a minimum.
Remove the zero entries of a minimizer and work on its positive support.
The support has at least three entries, because a vector supported on at
most two entries has second moment at least $1/2$. Its positive entries
cannot all be equal, since no reciprocal of an integer lies strictly
between $1/3$ and $1/2$. Therefore the gradients of the two constraints
are independent on the positive support.

Lagrange multipliers give numbers $\lambda,\mu$ such that every positive
entry satisfies
\[
 3x_\alpha^2-2\mu x_\alpha-\lambda=0.
\]
Hence there are exactly two distinct positive values, say $x>y>0$, and
$2\mu=3(x+y)$. The smaller value occurs only once. Indeed, if two
positions had value $y$, the vector with entries $1,-1$ in those
positions and zero elsewhere would be tangent to both constraints. For
the Lagrangian
$\sum x_\alpha^3-\lambda(\sum x_\alpha-1)
 -\mu(\sum x_\alpha^2-u)$, its Hessian on this vector is
\[
 2(6y-2\mu)=6(y-x)<0,
\]
contradicting the second-order necessary condition at a minimum. The
constraint gradients are independent, so the constraint set is a smooth
manifold near the minimizer and sufficiently short tangent curves preserve
positivity of all support entries.

If $x$ occurs $k$ times, then $kx+y=1$ and $0<y<x$ imply
$0<y<1/(k+1)$. The function
\[
 \frac{(1-y)^2}{k}+y^2
\]
is strictly decreasing on this interval, with endpoint values $1/k$ and
$1/(k+1)$. Thus $1/(k+1)<u<1/k$, so $k=2$. Put
$r=\sqrt{6u-2}$. Up to zero padding, the minimizing vector is
\[
 x,x,y,\qquad x=\frac{2+r}{6},\qquad y=\frac{1-r}{3}.
\]
A direct calculation gives
\[
 u=\frac{2+r^2}{6},\qquad
 v=2x^3+y^3=\frac{4+6r^2-r^3}{36},\qquad
 u-v=\frac{8+r^3}{36}=R(u).
\]
This proves \eqref{eq:sharp-moment-new} for $1/3<u<1/2$. At
$u=1/3$ and $u=1/2$, the coarse estimate gives the same endpoint values
$R(1/3)=2/9$ and $R(1/2)=1/4$. The argument applies to every finite
support size and hence to every finite codimension.

For $1/3<u<1/2$, differentiation gives
\[
 R'=\frac r4,\qquad R''=\frac{3}{4r},\qquad
 2(R')^2-RR''=\frac{5r^3-8}{48r}<0.
\]
Since $R>0$,
\[
 \left(\frac1R\right)''
 =\frac{2(R')^2-RR''}{R^3}<0.
\]
Thus $1/R$ is concave on $(1/3,1/2)$, and by continuity on the closed
interval as well.

\end{proof}

\subsection{Verification of the \texorpdfstring{$n=3,4,5$}{n=3,4,5}  parameter choices}\label{subsec:small-dimensions}
We verify the first three parameter choices in
Proposition~\ref{prop:normal-tangents}. For each of them,
\[
 \beta>\frac{11}{4},\qquad 2\beta+\eta>6,\qquad
 \alpha<\frac{1+\beta}{3},\qquad \alpha-\beta<-1.
\]
For reference, the two margins not involving $\eta$ are
\[
\begin{array}{c|cc}
\text{dimension used}&(1+\beta)/3-\alpha&\beta-\alpha-1\\ \hline
n=3&2/225&77/75\\
n=4&1/200&201/200\\
n=5&1/300&97/100
\end{array}
\]
The remaining two inequalities follow immediately from the parameters in
Proposition~\ref{prop:normal-tangents}.

Normalize $S=1$ as in Lemma~\ref{lem:normalized-scalar-reduction}. If
$D\ge1$, then Lemma~\ref{lem:normalized-scalar-reduction} gives
$\mathscr P\ge-1$ and $J\le1$, so
\[
 \mathscr P+\eta(1-J)\ge-1>
 \alpha-\beta\ge\alpha-\beta D.
\]
Thus only $0\le D<1$ remains. Put
\[
 a=\frac{1-D}{2},\qquad\text{so that}\qquad D=1-2a.
\]
Lemma~\ref{lem:normalized-scalar-reduction} gives
\[
 \mathscr P\ge u-1+\frac{2a^2}{u-v},\qquad
 1-J=u-a,\qquad
 0<a\le\min\{u,1/2\}.
\]
Consequently,
\[
 \mathscr P+\eta(1-J)
 \ge u-1+\eta(u-a)+\frac{2a^2}{u-v}.
\]
On the other hand,
\[
 \alpha-\beta D=\alpha-\beta+2\beta a.
\]
Therefore, to prove
$\mathscr P+\eta(1-J)\ge\alpha-\beta D$, it suffices to show
\begin{equation*}
 u-1+\eta(u-a)+\frac{2a^2}{u-v}
 -\alpha+\beta-2\beta a\ge0,
 \qquad 0<a\le\min\{u,1/2\}.
\end{equation*}
We now use Lemma~\ref{lem:moment-bounds} to reduce the verification to
three cases.  On the low range $0<a\le u\le1/3$, we use the coarse
bound \eqref{eq:coarse-moment-new} directly.  On the middle interval
\[
 \max\{a,1/3\}\le u\le1/2,
\]
the sharp bound \eqref{eq:sharp-moment-new}, together with the concavity
of $1/R$, shows that the corresponding lower bound is minimized at an
endpoint.  If $a\ge1/3$, the left endpoint is $u=a$; if $a\le1/3$, it is
$u=1/3$, which is already covered by the low range.  The right endpoint
is always $u=1/2$.  Finally, on $1/2\le u<1$, the lower bound obtained
from \eqref{eq:coarse-moment-new} is increasing in $u$, since its
$u$-derivative is
\[
 1+\eta+\frac{2a^2(2u-1)}{u^2(1-u)^2}>0.
\]
Thus this range is also controlled by $u=1/2$, and it remains to verify
the following three cases.

\smallskip
\noindent\emph{Case 1: $0<a\le u\le1/3$.}
By \eqref{eq:coarse-moment-new}, the normalized residual is bounded below by
\[
 u-1+\eta(u-a)+\frac{2a^2}{u(1-u)}
 -\alpha+\beta-2\beta a.
\]
For fixed $u$, its derivative with respect to $a$ is
\[
 \frac{4a}{u(1-u)}-(2\beta+\eta)
 \le\frac4{1-u}-(2\beta+\eta)
 \le6-(2\beta+\eta)<0.
\]
Hence the minimum over $0<a\le u$ occurs at $a=u$. There the lower bound
becomes
\[
 u-1+\frac{2u}{1-u}-\alpha+\beta-2\beta u,
\]
and its derivative with respect to $u$ satisfies
\[
 1-2\beta+\frac2{(1-u)^2}
 \le\frac{11}{2}-2\beta<0.
\]
Therefore the minimum in this range is attained at $a=u=1/3$ and equals
\[
 \frac{1+\beta}{3}-\alpha>0.
\]
This also verifies the endpoint $u=1/3$ when $a\le1/3$.

\smallskip
\noindent\emph{Case 2: $u=a\in[1/3,1/2]$.}
Put $r=\sqrt{6a-2}\in[0,1]$. By \eqref{eq:sharp-moment-new},
\[
 a-1+\frac{2a^2}{R(a)}
 =7a-2-\frac{r^3(1-r)^2}{8+r^3}.
\]
After subtracting the target $\alpha-\beta+2\beta a$ and multiplying by
the positive factor $8+r^3$, it remains to prove
\begin{equation}\label{eq:small-middle-residual}
 \left[\frac{\beta-7/2}{3}(1-r^2)+\frac32-\alpha\right]
 (8+r^3)-r^3(1-r)^2\ge0.
\end{equation}
We now substitute the three parameter choices.

For $n=3$, the left-hand side of
\eqref{eq:small-middle-residual} equals
\[
 (1-r)\left[\frac2{225}(1+r)(8+r^3)-r^3(1-r)\right].
\]
On $0\le r\le1/2$, the bracket is at least
\[
 \frac{16}{225}-\frac1{16}=\frac{31}{3600}>0,
\]
whereas on $1/2\le r\le1$ it is at least
\[
 \frac{24}{225}-\frac{27}{256}=\frac{23}{19200}>0.
\]
Indeed, $[r^3(1-r)]'=r^2(3-4r)$, so the relevant maxima are $1/16$
and $27/256$. Thus the numerator is positive for $r<1$ and vanishes
only at $r=1$.

For $n=4$, the left-hand side of
\eqref{eq:small-middle-residual} satisfies
\[
 \frac{8+r^3}{200}-r^3(1-r)^2
 \ge\frac8{200}-\frac{108}{3125}
 =\frac{17}{3125}>0,
\]
because $r^3(1-r)^2$ has its maximum $108/3125$ at $r=3/5$.

For $n=5$, it satisfies
\[
\begin{split}
 &\frac{(1+5r^2)(8+r^3)}{300}-r^3(1-r)^2\\
 &\qquad\ge\frac2{75}+\frac2{15}r^2-\frac4{27}r^2
 =\frac2{75}-\frac2{135}r^2
 \ge\frac8{675}>0,
\end{split}
\]
since $r(1-r)^2\le4/27$ on $[0,1]$.

\smallskip
\noindent\emph{Case 3: $u=1/2$.}
Using $u-v\le1/4$, the normalized residual is bounded below by
\[
 8a^2-(\eta+2\beta)a-\frac12+\frac\eta2-\alpha+\beta.
\]
For the three parameter choices, completing the square gives
\[
\begin{array}{c|c}
\text{dimension used}&\text{residual at }u=1/2\\ \hline
n=3&8(a-1/2)^2\\
n=4&8(a-19/40)^2\\
n=5&8(a-37/80)^2+7/800
\end{array}
\]
so this endpoint is nonnegative in every case.

It remains only to record the strictness for the $n=3$ parameters. If
$0<D<1$, then $a<1/2$. Case~1 is already strict; in Case~2 the only
possible zero is $r=1$, equivalently $a=1/2$; and in Case~3 the only
possible zero is again $a=1/2$. Hence the normalized residual is strictly
positive whenever $0<D<1$. The case $D\ge1$ was already strict above,
while for $D=0$ all estimates remain nonnegative. Restoring the degree-six
homogeneity proves the parameter choices for $n=3,4,5$ in
\eqref{eq:normal-tangent}, including the asserted strictness for $n=3$.

\subsection{Verification for \texorpdfstring{$n\ge6$}{n >= 6} parameter choices}\label{subsec:tail}
We now verify the fourth parameter choice
\[
 (\alpha,\beta,\eta)=\left(\frac{13}{9},\frac{10}{3},0\right).
\]
Equivalently, it remains to establish
\begin{equation}\label{eq:tail-tangent-new}
 \mathscr P\ge\frac{13}{9}S^3-\frac{10}{3}SD.
\end{equation}
Normalize $S=1$ and use Lemma~\ref{lem:normalized-scalar-reduction}.
If $D\ge1$, then
\[
 \frac{13}{9}-\frac{10}{3}D\le-\frac{17}{9}<-1\le\mathscr P,
\]
so only $D<1$ remains. Put
\[
 a=\frac{1-D}{2},\qquad D=1-2a.
\]
Lemma~\ref{lem:normalized-scalar-reduction} gives
\[
 \mathscr P\ge u-1+\frac{2a^2}{u-v},
 \qquad 0<a\le\min\{u,1/2\}.
\]
On the other hand, the right-hand side of \eqref{eq:tail-tangent-new}, after
normalizing $S=1$, becomes
\[
 \frac{13}{9}-\frac{10}{3}D
 =\frac{13}{9}-\frac{10}{3}(1-2a)
 =-\frac{17}{9}+\frac{20}{3}a.
\]
Consequently, it is enough to prove
\[
 u-1+\frac{2a^2}{u-v}
 \ge -\frac{17}{9}+\frac{20}{3}a,
\]
or, equivalently,
\begin{equation}\label{eq:tail-scalar-target}
 u+\frac89-\frac{20}{3}a+\frac{2a^2}{u-v}\ge0,
 \qquad0<a\le\min\{u,1/2\}.
\end{equation}
Here it is more convenient to fix $u$ and minimize the resulting convex
quadratic in $a$.

\smallskip
\noindent\emph{Case 1: $u\le1/3$.}
By \eqref{eq:coarse-moment-new}, the left-hand side of
\eqref{eq:tail-scalar-target} is bounded below by
\[
 u+\frac89-\frac{20}{3}a+\frac{2a^2}{u(1-u)}.
\]
As a quadratic in $a$, its unconstrained minimizer is
\[
 a_*=\frac53u(1-u)\ge u.
\]
Since the admissible interval is $0<a\le u$, the constrained minimum is
attained at $a=u$ and equals
\[
 \frac{(3u-1)(17u-8)}{9(1-u)}\ge0,
\]
because both numerator factors are nonpositive when $u\le1/3$.

\smallskip
\noindent\emph{Case 2: $u\ge1/2$.}
Using \eqref{eq:coarse-moment-new} again, the same quadratic has global
minimum
\[
 \frac{(2u-1)(25u-8)}9\ge0.
\]
Therefore its minimum on the smaller admissible interval
$0<a\le1/2$ is also nonnegative.

\smallskip
\noindent\emph{Case 3: $1/3\le u\le1/2$.}
Put $r=\sqrt{6u-2}\in[0,1]$. By \eqref{eq:sharp-moment-new}, it suffices
to minimize
\[
 u+\frac89-\frac{20}{3}a+\frac{2a^2}{R(u)}
\]
over $0<a\le u$. Its unconstrained minimizer is
\[
 a_*=\frac53R(u).
\]
There are two possibilities.

If $a_*\le u$, equivalently
\[
 5r^3-18r^2+4\le0,
\]
then $a_*$ is admissible. This condition forces $r>1/2$, because the
polynomial on the left is decreasing on $[0,1/2]$ and equals $1/8$ at
$r=1/2$. Substituting $a=a_*$ gives the minimum value
\[
 \frac{(1-r)(25r^2-2r-2)}{162}\ge0.
\]
Indeed, $25r^2-2r-2$ is increasing for $r\ge1/2$ and equals $13/4$
at $r=1/2$.

If $a_*>u$, the vertex lies to the right of the admissible interval, so
the constrained minimum occurs at $a=u$. Direct substitution gives
\[
 \frac{r^2(8-18r+36r^2-17r^3)}{18(8+r^3)}\ge0.
\]
To see the last inequality, note that for $0\le r\le1$,
\[
 8-18r+36r^2-17r^3
 \ge8-18r+19r^2>0,
\]
and the final quadratic has discriminant $-284$. This also covers the
endpoints $u=1/3$ and $u=1/2$.

The three cases prove \eqref{eq:tail-scalar-target}. Restoring homogeneity
proves \eqref{eq:tail-tangent-new} and completes the proof of
Proposition~\ref{prop:normal-tangents}.

\subsection{The trace-free refinement}\label{subsec:tracefree}

\begin{proof}[\textbf{Proof of Proposition~\ref{prop:tracefree-tangent}}]
If $S=0$, every matrix vanishes and the assertion is immediate. Assume
$S>0$ and normalize $S=1$ as in
Lemma~\ref{lem:normalized-scalar-reduction}. We keep the notation
$x_\alpha,u,v,U_{\alpha\beta},p,J,D$ from that lemma after diagonalizing
Lu's fundamental matrix $\mathcal A$ and omitting its zero eigenvalues.
The retained matrices remain trace-free because orthogonal normal
combinations preserve trace.

\smallskip
\noindent\emph{Step 1: the trace-free eigenvalue estimate.}
Let $A$ be a trace-free symmetric $n\times n$ matrix with eigenvalues
$\lambda_1,\ldots,\lambda_n$. For $i\ne j$, trace-freeness gives
$\lambda_i+\lambda_j=-\sum_{k\ne i,j}\lambda_k$, and hence
Cauchy--Schwarz yields
\[
 (\lambda_i+\lambda_j)^2
 \le(n-2)\sum_{k\ne i,j}\lambda_k^2.
\]
Equivalently,
\[
 n(\lambda_i+\lambda_j)^2
 +(n-2)(\lambda_i-\lambda_j)^2
 \le2(n-2)|A|^2.
\]
Using
$4\lambda_i\lambda_j=(\lambda_i+\lambda_j)^2
-(\lambda_i-\lambda_j)^2$, we obtain
\begin{equation}\label{eq:tracefree-eigenvalues}
 \lambda_i\lambda_j
 \le\frac{n-2}{2n}|A|^2
    -\frac{n-1}{2n}(\lambda_i-\lambda_j)^2.
\end{equation}
This is the only point in the argument where trace-freeness is used.

\smallskip
\noindent\emph{Step 2: the refined contraction estimate.}
Write $C_{\alpha\beta}=[A_\alpha,A_\beta]$ and fix
$\alpha<\beta$. Diagonalize $A_\alpha$ and write
$A_\beta=(b_{ij})$. Since
$(C_{\alpha\beta})_{ij}=(\lambda_i-\lambda_j)b_{ij}$,
\eqref{eq:tracefree-eigenvalues} gives
\[
\begin{split}
 \langle A_\alpha C_{\alpha\beta}A_\alpha,C_{\alpha\beta}\rangle
 &\le\frac{n-2}{2n}x_\alpha U_{\alpha\beta}
   -\frac{n-1}{2n}\sum_{i,j}(\lambda_i-\lambda_j)^4b_{ij}^2\\
 &\le\frac{n-2}{2n}x_\alpha U_{\alpha\beta}
   -\frac{n-1}{2n}\frac{U_{\alpha\beta}^2}{x_\beta}.
\end{split}
\]
The last step is Cauchy--Schwarz, using
$\sum_{i,j}b_{ij}^2=x_\beta>0$. Interchanging $\alpha$ and $\beta$
gives the analogous estimate with $x_\alpha$ in the denominator. If
$\gamma\notin\{\alpha,\beta\}$, diagonalizing $A_\gamma$ and dropping
the nonpositive difference term in \eqref{eq:tracefree-eigenvalues}
gives
\[
 \langle A_\gamma C_{\alpha\beta}A_\gamma,C_{\alpha\beta}\rangle
 \le\frac{n-2}{2n}x_\gamma U_{\alpha\beta}.
\]
Since $\sum_\gamma x_\gamma=1$, the contribution of this fixed unordered
pair satisfies
\[
\begin{split}
 \sum_\gamma
 \langle A_\gamma C_{\alpha\beta}A_\gamma,C_{\alpha\beta}\rangle
 &\le\frac{n-2}{2n}U_{\alpha\beta}\\
 &\quad-\frac{n-1}{2n}
 \frac{(x_\alpha+x_\beta)U_{\alpha\beta}^2}
 {x_\alpha x_\beta}.
\end{split}
\]
Now use the exact identity \eqref{eq:normal-contraction-new}. Both ordered
pairs $(\alpha,\beta)$ and $(\beta,\alpha)$ occur in its triple sum, so
the prefactor $-2$ produces a factor $-4$ for each unordered pair.
Because $J=2\sum_{\alpha<\beta}U_{\alpha\beta}$, we obtain
\begin{equation}\label{eq:tracefree-raw}
 \mathscr P\ge p-\frac{n-2}{n}J
 +\frac{2(n-1)}{n}\sum_{\alpha<\beta}
   \frac{(x_\alpha+x_\beta)U_{\alpha\beta}^2}{x_\alpha x_\beta}.
\end{equation}
Thus the trace-free information has been converted into a strengthened
scalar lower bound for $\mathscr P$.

\smallskip
\noindent\emph{Step 3: the case $D\ge1$.}
By \eqref{eq:tracefree-raw}, $p\ge0$, and
Lemma~\ref{lem:normalized-scalar-reduction},
\[
 \mathscr P\ge-\frac{n-2}{n}.
\]
On the other hand,
\[
 \alpha_n-\sqrt6D
 \le-\frac43+\frac{137}{144n}
 <-1+\frac2n=-\frac{n-2}{n},
\]
and the strict margin in the middle inequality is
$1/3+151/(144n)>0$. Hence the desired inequality is strict when
$D\ge1$. This also includes the rank-one case, for which
$u=1$, $J=p=0$, and $D=1$.

\smallskip
\noindent\emph{Step 4: the case $0\le D<1$.}
Put $a=(1-D)/2$. By
Lemma~\ref{lem:normalized-scalar-reduction},
\[
 0<a\le u<1,\qquad J=1+a-u,\qquad p\ge a,
\]
and
\[
 \sum_{\alpha<\beta}
 \frac{(x_\alpha+x_\beta)U_{\alpha\beta}^2}{x_\alpha x_\beta}
 \ge\frac{p^2}{u-v}.
\]
Moreover, Lemma~\ref{lem:moment-bounds} gives
$u-v\le u(1-u)$. Since the right-hand side of
\eqref{eq:tracefree-raw} is increasing as a function of $p\ge0$, these
estimates yield
\begin{equation}\label{eq:tracefree-au-reduction}
 n\mathscr P\ge
 2a+(n-2)(u-1)+\frac{2(n-1)a^2}{u(1-u)}.
\end{equation}
All denominators are positive because $0<u<1$.

We first complete the square in $a$. Since $D=1-2a$, adding
$n\sqrt6D$ to \eqref{eq:tracefree-au-reduction} gives
\begin{equation}\label{eq:tracefree-square-a}
\begin{split}
 n(\mathscr P+\sqrt6D)
 &\ge\frac{2(n-1)}{u(1-u)}
   \left(a-\frac{(n\sqrt6-1)u(1-u)}{2(n-1)}\right)^2\\
 &\quad+(n-2)(u-1)+n\sqrt6
   -\frac{(n\sqrt6-1)^2}{2(n-1)}u(1-u).
\end{split}
\end{equation}
The square is nonnegative, so no information about the location of its
vertex is needed. For $n\ge3$,
\[
\begin{split}
 \frac{(n\sqrt6-1)^2}{2(n-1)}
 &=3n+3-\sqrt6+\frac{7/2-\sqrt6}{n-1}\\
 &\le3n+\frac{19}{4}-\frac{3\sqrt6}{2}.
\end{split}
\]
Set
\[
 c_0:=\frac{19}{4}-\frac{3\sqrt6}{2}>0.
\]
Discarding the square in \eqref{eq:tracefree-square-a}, using the last
bound, and then $u(1-u)\le1/4$, we obtain
\begin{equation}\label{eq:tracefree-square-u}
\begin{split}
 &n\left(\mathscr P+\sqrt6D-\sqrt6+\frac43\right)\\
 &\quad\ge
 3n\left(u-\frac13\right)^2
 -2\left(u-\frac13\right)+\frac43-c_0u(1-u)\\
 &\quad\ge
 3n\left(u-\frac13\right)^2
 -2\left(u-\frac13\right)+\frac43-\frac{c_0}{4}.
\end{split}
\end{equation}
We now complete the square in $u$:
\[
 3n\left(u-\frac13\right)^2-2\left(u-\frac13\right)
 =3n\left(u-\frac13-\frac1{3n}\right)^2-\frac1{3n}.
\]
Since $n\ge3$, \eqref{eq:tracefree-square-u} therefore gives
\[
\begin{split}
 n\left(\mathscr P+\sqrt6D-\sqrt6+\frac43\right)
 &\ge\frac43-\frac1{3n}-\frac{c_0}{4}\\
 &\ge\frac5{144}+\frac{3\sqrt6}{8}
 >\frac{137}{144}.
\end{split}
\]
For the last strict inequality we use $\sqrt6>22/9$:
\[
 \frac5{144}+\frac{3\sqrt6}{8}
 >\frac5{144}+\frac{11}{12}
 =\frac{137}{144}.
\]
Consequently,
\[
 \mathscr P>
 \sqrt6-\frac43+\frac{137}{144n}-\sqrt6D
 =\alpha_n-\sqrt6D
\]
whenever $0\le D<1$.

Combining Steps~3 and~4 proves the normalized inequality, with strictness
for every nonzero tuple. Restoring the original scale gives
\eqref{eq:tracefree-tangent}. As in
Lemma~\ref{lem:normalized-scalar-reduction}, repeated positive
eigenvalues of the fundamental matrix cause no difficulty because the
normal diagonalization is only pointwise and the chosen frame is never
differentiated.
\end{proof}

\medskip
\noindent\textbf{Declaration of generative AI use.}
During the preparation and revision of this manuscript, the author used OpenAI's ChatGPT, including its Codex tools, as an assistive tool for English-language polishing, \LaTeX{} formatting, symbolic and numerical computations, internal-consistency checks, and the exploration and verification of certain algebraic inequality estimates. AI-generated suggestions and computations were used only as auxiliary input and were independently checked before being incorporated into the manuscript. All arguments, proofs, constants, and conclusions appearing in the final version were independently verified by the author, who assumes full responsibility for the mathematical content.

\begin{acknow}
The authors are grateful to Prof.  Jianquan Ge and Dr. Weiran Ding for their encouragement, support, and valuable discussions.
\end{acknow}

\end{document}